\documentclass[12pt]{amsart}

\usepackage{cancel}
\usepackage{latexsym, amssymb, amsmath,esint}
\usepackage{soul}
\usepackage{amsfonts, graphicx}
\usepackage{graphicx,color}
\newcommand{\be}{\begin{equation}}
\newcommand{\ee}{\end{equation}}
\newcommand{\beq}{\begin{eqnarray}}
\newcommand{\eeq}{\end{eqnarray}}

\usepackage{wasysym,stmaryrd}
\newtheorem{thm}{Theorem}[section]

\newtheorem{lma}[thm]{Lemma}
\newtheorem{prop}[thm]{Proposition}
\newtheorem{cor}[thm]{Corollary}
\newtheorem{defn}[thm]{Definition}
\newtheorem{ques}[thm]{Question}
\theoremstyle{remark}
\newtheorem{rem}[thm]{Remark}
\numberwithin{equation}{section}

\def\be{\begin{equation}}
\def\ee{\end{equation}}
\def\bee{\begin{equation*}}
\def\eee{\end{equation*}}

\def\e{\varepsilon}

\def\a{{\alpha}}

\begin{document}

\title[]
{Llarull's theorem on punctured sphere with $L^\infty$ metric}

\author{Jianchun Chu}
\address[Jianchun Chu]{School of Mathematical Sciences, Peking University, Yiheyuan Road 5, Beijing 100871, People's Republic of China}
\email{jianchunchu@math.pku.edu.cn}

 \author{Man-Chun Lee}
\address[Man-Chun Lee]{Department of Mathematics, The Chinese University of Hong Kong, Shatin, Hong Kong, China}
\email{mclee@math.cuhk.edu.hk}

 \author{Jintian Zhu}
\address[Jintian Zhu]{Institute for Theoretical Sciences, Westlake University, 600 Dunyu Road, 310030, Hangzhou, Zhejiang, People’s Republic of China}
\email{zhujintian@westlake.edu.cn}

\renewcommand{\subjclassname}{\textup{2020} Mathematics Subject Classification}

\date{\today}

\begin{abstract}
The classical Llarull theorem states that a smooth metric on $n$-sphere cannot have scalar curvature no less than $n(n-1)$ and dominate the standard spherical metric at the same time unless it is the standard spherical metric. In this work, we prove that Llarull's rigidity theorem holds for $L^{\infty}$ metrics on spheres with finitely many points punctured. This is related to a question of Gromov.

\end{abstract}

\maketitle

\markboth{Jianchun Chu, Man-Chun Lee, Jintian Zhu}{Singular Llarull's Theorem}

\section{Introduction}\label{introduction}

In recent years, a lot of progress has been made in understanding the scalar curvature. We refer readers to Gromov's lecture note \cite{Gromov} on scalar curvature for a comprehensive overview. When the model is the standard sphere, Llarull \cite{Llarull} showed that if a spin manifold $(M^n,g)$ has $\mathrm{scal}_g\geq n(n-1)$ and admits an area non-increasing smooth map $f$ with non-zero degree to the standard sphere $\mathbb{S}^n$, then the map must be an isometry. The method is based on Dirac operators, and is inspired by the fundamental work of Gromov-Lawson \cite{GromovLawson}.

There has been many progress in generalizing Llarull's theorem in many different aspects. In \cite{GoetteSemmelmann}, Goette-Semmelmann showed that the rigidity result still holds when the sphere is replaced by closed manifolds with non-zero Euler characteristic and non-negative curvature operator. This result was generalized by Lott \cite{Lott} to compact manifolds with boundary. Along this direction, Cecchini-Zeidler \cite{CecchiniZeidler} extended the rigidity for odd-dimensional manifolds with a warped product metric. Later, B\"ar-Brendle-Hanke-Wang \cite{BarBrendleHankeWang} and Wang-Xie \cite{WangXie} generalized this result to all dimensions independently.
It is also worth mentioning that in dimension 4, the distance non-increasing version of Llarull's theorem was recently proved by Cecchini-Wang-Xie-Zhu \cite{CecchiniWangXiuZhu} without the spin assumption.

Recently, there is an increasing interest in understanding the rigidity in scalar curvature on metrics with weaker regularity. For example, it was asked by Gromov \cite{Gromov} if Llarull's theorem also holds in the sense of distance isometry if the non-zero degree map $f$ from the closed spin manifold $M^n$ to $\mathbb{S}^n$ is only $1$-Lipschitz. This was answered affirmatively by Cecchini-Hanke-Schick \cite{CecchiniHankeSchick} and Lee-Tam \cite{LeeTam2} independently using different methods. In this work, we are interested in the generalization in another direction. In \cite[Section 3.9]{Gromov}, Gromov asked if Llarull's theorem holds for metrics that are defined on the sphere with some subset removed.
\begin{ques}
Let $(\mathbb{S}^{n},g_{\mathrm{sph}})$ be the standard sphere and $\mathcal{S}$ be a subset of $\mathbb{S}^{n}$. Suppose that $g$ is a smooth metric on $\mathbb{S}^{n}\setminus\mathcal{S}$ such that
\begin{enumerate}\setlength{\itemsep}{1mm}
\item[(a)] $\mathrm{scal}_g\geq \mathrm{scal}_{g_{\mathrm{sph}}}=n(n-1)$;
\item[(b)] $g\geq g_{\mathrm{sph}}$
\end{enumerate}
on $M\setminus \mathcal{S}$. Which conditions on $\mathcal{S}$ can ensure that $g=g_{\mathrm{sph}}$?
\end{ques}

When $\mathcal{S}$ consists of two antipodal points, Gromov gave a sketched argument in \cite[Section 5.5 and 5.7]{Gromov}. 
When $n=3$, it was studied by  Hu-Liu-Shi \cite{HuLiuShi} using $\mu$-bubbles method and Hirsch-Kazaras-Khuri-Zhang \cite{HirschKazarasKhuriZhang}  using space-time harmonic function. The rigidity with two antipodal points removed was further generalized by B\"ar-Brendle-Hanke-Wang \cite{BarBrendleHankeWang} and Wang-Xie \cite{WangXie}  to all dimensions using spin method. Their method are all based on the warped product interpretation of $\mathbb{S}^n\setminus \{\pm x_0\}$. We might interpret this as a removable singularities result.

The removable singularities in view of scalar curvature rigidity is also related to another question of Schoen, for example see \cite[Conjecture 1.5]{LiMantoulidis}. To state it properly, let us recall the definition of $L^{\infty}$ metric first:

\begin{defn}
We say that $g$ is an $L^\infty$ metric on a closed manifold $M$ if it is a measurable sections of $\mathrm{Sym}_2(T^*M)$ such that $\Lambda^{-1}g_0\leq g\leq \Lambda g_0$ almost everywhere on $M$ for some smooth metric $g_0$ on $M$ and constant $\Lambda>1$.
\end{defn}

When the model space is the standard torus, Schoen's question can be stated as follows:

\begin{ques}
If an $L^\infty$ metric $g$ on torus $\mathbb{T}^n$ is smooth outside a closed embedded sub-manifold $\mathcal{S}$ with co-dimension greater than or equal to $3$. If $\mathrm{scal}_g\geq 0$ on $\mathbb{T}^n\setminus\mathcal{S}$, then $g$ is a smooth flat metric.
\end{ques}

We refer interested readers to \cite{LiMantoulidis,Kazaras,JiangShengZhang,LeeTam1,ChuLee} for some recent progress on this problem. Motivated by the questions of Gromov and Schoen, we consider $L^{\infty}$ metrics on spheres with finitely many points punctured. 
More generally, we prove the following:
\begin{thm}\label{intro-main}
Let $M^n$ be a closed spin manifold with $n\geq 3$ and $\mathcal{S}=\{p_i\}_{i=1}^N$ be a finite set in $M$. 
Suppose that $f:M\to \mathbb{S}^n$ is a Lipschitz map with non-zero degree and $g$ is an $L^\infty$ metric so that both $f,g$ are smooth outside $\mathcal{S}$ and satisfy
\begin{enumerate}\setlength{\itemsep}{1mm}
\item[(a)] $\mathrm{scal}_g\geq \mathrm{scal}_{g_{\mathrm{sph}}}=n(n-1)$;
\item[(b)] $f$ is area non-increasing, i.e. $|\Lambda^2 \mathrm{d}f|\leq 1$
\end{enumerate}
on $M\setminus \mathcal{S}$, 
then $f$ is a distance isometry and satisfies $f^*g_{\mathrm{sph}}=g$ outside $\mathcal{S}$.
\end{thm}

In particular, if we take $M=\mathbb{S}^{n}$ and $f=\mathrm{id}$, then it reduces back to the setting of Gromov's question with additional $L^\infty$ assumption of metric across the point singularities $\mathcal{S}$.

\begin{cor}
Suppose that $g$ is an $L^{\infty}$ metric on $\mathbb{S}^{n}\setminus\mathcal{S}$ for some finite subset $\mathcal{S}$ such that
\begin{enumerate}\setlength{\itemsep}{1mm}
\item[(a)] $\mathrm{scal}_g\geq \mathrm{scal}_{g_{\mathrm{sph}}}=n(n-1)$;
\item[(b)] $g\geq g_{\mathrm{sph}}$
\end{enumerate}
on $M\setminus \mathcal{S}$. Then $g=g_{\mathrm{sph}}$.
\end{cor}

Our method is largely motivated by the strategy of Cecchini-Wang-Xie-Zhu \cite{CecchiniWangXiuZhu} and smoothing method of Li-Mantoulidis \cite{LiMantoulidis}. Inspired by the work of Li-Mantoulidis, we smooth out the point singularities $\mathcal{S}$ using the Green function which open up the compact manifold to a complete non-compact manifold. In \cite{LiMantoulidis}, gluing  using minimal surface is performed so that the complete non-compact manifold is compactified with topology unchanged overall. This relies heavily on the classification of minimal surfaces and thus the method is restrictive to dimension three. In contrast, we work on the complete non-compact manifold directly. By constructing a smooth non-zero degree map from the opened non-compact manifold to the standard sphere, we derive a contradiction using conformal method in a similar spirit of \cite{CecchiniWangXiuZhu}.

\vskip0.2cm

{\it Acknowledgement:} The first-named author was partially supported by National Key R\&D Program of China 2023YFA1009900 and NSFC grant 12271008. The second-named author was partially supported by Hong Kong RGC grant (Early Career Scheme) of Hong Kong No. 24304222, No. 14300623, and a NSFC grant No. 12222122. The third-named author was partially supported by National Key R\&D Program of China with grant no. 2023YFA1009900 as well as the startup fund from Westlake University.

\section{Conformal deformation}\label{sec:conf}

In this section, we discuss the scalar curvature improvement using the conformal deformation. Let $M^n$ be a closed manifold with $n\geq 3$ and $\mathcal{S}=\{p_i\}_{i=1}^N$ be a finite set in $M$. 
Suppose that $f:M\to \mathbb{S}^n$ is a Lipschitz map and $g$ is an $L^\infty$ metric so that both $f,g$ are smooth outside $\mathcal{S}$.
We consider the scalar curvature lower bound in term of total area defect. 
For $x\in M\setminus\mathcal{S}$, we denote
\begin{equation}
\sigma_{f,g}(x)=\sum_{i\neq j}\mu_i \mu_j=2\sum_{i<j}\mu_i\mu_j,
\end{equation}
where $\mu_i\geq 0$ are singular values of the linear map 
$$\mathrm df_x:(T_xM,g_x)\to (T_{f(x)}\mathbb S^n,g_{\mathbb S^n,f(x)}).$$ We might omit the index $g$ if the content is clear. 
The map $f$  is said to be distance non-increasing if $\mu_i\leq 1$ for all $i$, and area non-increasing if $\mu_i\mu_j\leq 1$ for each $i\neq j$.

We consider the corresponding conformal deformation. Unlike the smooth case, the scalar curvature might be blowing up nearby the singularity $\mathcal{S}$. So we start with introducing the truncation function. Let $\phi:\mathbb{R}\to \mathbb{R}$ be a smooth non-decreasing function so that
\[
\phi(s)=
\begin{cases}
\, s & \mbox{if $s\leq10$}; \\
\, 20 & \mbox{if $s\geq30$}.
\end{cases}
\]
We consider the operator
\begin{equation*}
\mathcal{L}_{g}=-\frac{4(n-1)}{n-2}\Delta_g+\phi(\mathrm{scal}_g-\sigma_f)
\end{equation*}
on $M$. Notice that the zeroth order term $\phi(\mathrm{scal}_g-\sigma_f)$ is bounded from above. 

\begin{lma}\label{lma:1eigen}
Under the above set-up, suppose the metric $g$ and map $f$ satisfy
\begin{enumerate}\setlength{\itemsep}{1mm}
\item[(i)]  $\mathrm{scal}_g\geq \sigma_{f}$ on $M\setminus \mathcal{S}$;
\item[(ii)] $\mathrm{scal}_g(x_0)> \sigma_{f}(x_0)$ for some $x_0\in M\setminus \mathcal{S}$,
\end{enumerate}
then there exists $u\in C^{\infty}_{\mathrm{loc}}(M\setminus \mathcal{S})\cap C^\a(M)$ for some $\a\in (0,1)$ such that
$$\mathcal{L}_{g}u=\lambda_1 u$$
on $M\setminus \mathcal{S}$ and $C_0^{-1}\leq u\leq C_0$ on $M$ for some constant $\lambda_1,C_{0}>0$.
\end{lma}
\begin{proof}
The argument is analogous to \cite[Lemma 4.1]{LiMantoulidis} with suitable modifications. First notice that we have $0\leq\phi(\mathrm{scal}_g-\sigma_f)\leq20$
from assumption (i) and the definition of the truncation function $\phi$. In particular, $\phi(\mathrm{scal}_g-\sigma_f)\in L^\infty(M,g)$. 
Then for each $v\in W^{1,2}(M,g)$ we can define the functional 
\begin{equation*}
\mathcal{I}(v)=\int_M c_n |\nabla v|_g^2+\phi(\mathrm{scal}_g-\sigma_{f}) v^2 \,d\mathrm{vol}_g,
\end{equation*}
where 
$c_n=\frac{4(n-1)}{n-2}$. Set
\begin{equation*}
\lambda_1:= \inf\{ \mathcal{I}(v): v\in W^{1,2}(M,g),\, \|v\|_{L^2(M,g)}=1 \}.
\end{equation*}
Since $g$ is an $L^\infty$ metric, it follows from standard elliptic PDE theory (for example see \cite[Section 8.12]{GilbargTrudinger}) that there exists a non-negative function $u\in W^{1,2}(M,g)$ such that $\|u\|_{L^{2}(M,g)}=1$ and $\mathcal{I}(u)=\lambda_{1}$. From the Euler-Lagrange equation of this variational problem, $u$ is a weak solution of $\mathcal{L}_gu=\lambda_{1}u$ on $M$.
The standard elliptic regularity theory (for example see \cite[Corollary 8.11, Theorem 8.22]{GilbargTrudinger}) shows that $u\in C^{\infty}_{\mathrm{loc}}(M\setminus \mathcal{S})\cap C^\a(M)$ for some $\alpha\in(0,1)$. By the Harnack inequality (for example see \cite[Theorem 8.20]{GilbargTrudinger}), we see that
$C_0^{-1}\leq u\leq C_0$ on $M$ for some constant $C_0>0$. From assumption (ii), it is clear that $\lambda_1>0$. This completes the proof.
\end{proof}

The next proposition shows that by the conformal transformation, we might assume working with metric such that scalar curvature inequality $\mathrm{scal}\geq\sigma_{f}$ is strict unless the equality holds globally.

\begin{prop}\label{prop:dicho}
Suppose that $f:M\to \mathbb{S}^n$ is a Lipschitz map and $g$ is an $L^\infty$ metric so that both $f,g$ are smooth outside $\mathcal{S}$ and satisfy $\mathrm{scal}_g\geq \sigma_{f,g}$ outside $\mathcal{S}$.
Then one of the following holds:
\begin{enumerate}\setlength{\itemsep}{1mm}
\item[(a)]  either $\mathrm{scal}_g\equiv \sigma_{f,g}$ on $M\setminus \mathcal{S}$;
\item[(b)] or there exists another $L^{\infty}$ metric $\tilde g$ on $M$ which is smooth outside $\mathcal{S}$ such that
$\mathrm{scal}_{\tilde g}\geq \sigma_{f,\tilde g}+\delta_1$ on $M\setminus \mathcal{S}$ for some constant $\delta_1>0$.
\end{enumerate}
\end{prop}

\begin{proof}
Suppose that (a) does not hold, then we have $\mathrm{scal}_g\geq \sigma_{f,g}$ on $M\setminus \mathcal{S}$ and $\mathrm{scal}_g(x_0)>\sigma_f(x_0)$ for some $x_0\in M\setminus \mathcal{S}$. We let $u$ be the function obtained from Lemma~\ref{lma:1eigen}. Recall that we have $\mathcal{L}_gu=\lambda_1u$, where $u$ satisfies $C_0^{-1}\leq u\leq C_0$ for some positive constant $C_0$ and $\lambda_1$ satisfies $\lambda_1\geq \delta$ for some positive constant $\delta$. Define $\tilde g=u^\frac{4}{n-2} g$. Then $\tilde{g}$ is an $L^{\infty}$ metric which is still smooth outside $\mathcal{S}$. 
On $M\setminus\mathcal{S}$, its scalar curvature satisfies
\begin{equation*}
\begin{split}
\mathrm{scal}_{\tilde g}
= {} & u^{-\frac{n+2}{n-2}}\left( -\frac{4(n-1)}{n-2}\Delta_g +\mathrm{scal}_g\right)u \\
\geq {} & u^{-\frac{n+2}{n-2}}\left( \mathcal{L}_g+\sigma_{f,g}\right)u \\[2mm]
\geq {} & (\delta+\sigma_{f,g})u^{-\frac{4}{n-2}} \\[3mm]
\geq {} & \delta_1 +\sigma_{f,\tilde g},
\end{split}
\end{equation*}
where $\delta_{1}=C_{0}^{-\frac{4}{n-2}}\delta$. This completes the proof.
\end{proof}

\begin{rem}
The targeting manifold $\mathbb{S}^n$ had indeed played no role here and thus can be replaced by any smooth manifolds.
\end{rem}

\section{Llarull's theorem with isolated singularity}
In this section, we will finish the proof of Theorem~\ref{intro-main}. We begin with some preliminary results.  
Write $\hat{\mathcal{S}}=f(\mathcal{S})$ and  $\hat p_i=f(p_{i})$.

\begin{lma}\label{lma:phi-delta}
Given any $\delta>0$, there exists a smooth map $\phi_\delta:(\mathbb{S}^n,g_{\mathrm{sph}})\to (\mathbb{S}^n,g_{\mathrm{sph}})$ satisfying
\begin{enumerate}\setlength{\itemsep}{1mm}
        \item[(i)] $\|\mathrm{d}\phi_\delta\|_{L^{\infty}} \leq 1+\delta$;
        \item[(ii)] for each point $\hat p\in \hat {\mathcal{S}}$, there are  open neighborhoods $$V_{\hat p,\delta}\subsetneq  W_{\hat  p,\delta}$$ of $\hat p$ such that $\phi_\delta(V_{\hat p,\delta})=\{\hat p\}$;
        \item[(iii)] $\phi_\delta(x)=x$ outside $W_{\hat p,\delta}$ for each $\hat  p\in \hat {\mathcal{S}}$.
    \end{enumerate}
\end{lma}
\begin{proof}
Fix a constant $r_{0}\in(0,\pi/2)$ such that the geodesic balls 
$$\{B_{g_{\mathrm{sph}}}(\hat p_{i},2r_{0})\}_{i=1}^{N}$$ are pairwise disjoint. Since $2r_{0}$ is less than $\pi$ (i.e. the injective radius of the unit sphere $(\mathbb{S}^{n},g_{\mathrm{sph}})$), then the exponential map $\mathrm{exp}_{\hat p_{i}}:B_{\mathbb{R}^{n}}(0,2r_{0})\to B_{g_{\mathrm{sph}}}(\hat p_{i},2r_{0})$ is a diffeomorphism. For any $x\in B_{g_{\mathrm{sph}}}(\hat p_{i},2r_{0})$, there exists $(t,\theta)\in[0,2r_{0})\times\mathbb{S}^{n-1}$ such that $x=\mathrm{exp}_{\hat p_{i}}(t\theta)$. Define a map $F_{i}:B_{g_{\mathrm{sph}}}(\hat p_{i},2r_{0})\to B_{g_{\mathrm{sph}}}(\hat p_{i},2r_{0})$ by
\[
F_{i}(x) = \mathrm{exp}_{\hat p_{i}}(\eta(t)\theta),
\]
where $\eta:[0,2r_{0}]\to[0,\infty)$ is a smooth function such that
\begin{itemize}\setlength{\itemsep}{1mm}
\item $\eta(t)=0$ for $t\in[0,\frac{\delta r_{0}}{2(1+\delta)}]$;
\item $\eta(t)=t$ for $t\in[r_{0},2r_{0}]$;
\item $0\leq \eta'(t)\leq 1+\delta$ for $t\in[0,2r_{0}]$.
\end{itemize}
It is clear that $|\mathrm dF_{i}|\leq1+\delta$ and $F_{i}$ is the identity map on $B_{g_{\mathrm{sph}}}(\hat p_{i},2r_{0})\setminus B_{g_{\mathrm{sph}}}(\hat p_{i},r_{0})$. Then the map $\phi_{\delta}:(\mathbb{S}^{n},g_{\mathrm{sph}})\to(\mathbb{S}^{n},g_{\mathrm{sph}})$ defined by
\[
\phi_{\delta}(x) :=
\begin{cases}
\, F_{i}(x), & \mbox{$x\in B_{g_{\mathrm{sph}}}(\hat p_{i},2r_{0})$;} \\[1mm]
\, x, & \mbox{$x\in\mathbb{S}^{n}\setminus \cup_{i=1}^NB_{g_{\mathrm{sph}}}(\hat p_{i},r_{0})$}
\end{cases}
\]
and the sets $V_{\hat p_{i}}:=B_{g_{\mathrm{sph}}}(\hat p_{i},\frac{\delta r_{0}}{2(1+\delta)})$, $W_{\hat p_{i}}:=B_{g_{\mathrm{sph}}}(\hat p_{i},r_{0})$ are the required map and neighborhoods.
\end{proof}

To handle the singular version of Llarull's theorem with $L^\infty$ metric, we need to find a way to desingularize the singularity $\mathcal S$ properly. The basic idea is to perform conformal deformation using Green type function such that the conformal metric becomes complete and the singularity $\mathcal S$ disappears as the infinity, where a variation of the Llarull theorem for complete manifolds (see Proposition \ref{prop:Zhang-noncpt}) can be applied.
\medskip

We need the following existence of Green type function.
\begin{lma}\label{lma:Green}
Under the set-up in Section~\ref{sec:conf}, suppose the metric $g$ and map $f$ satisfy  $\mathrm{scal}_g> \sigma_{f}$ on $M\setminus \mathcal{S}$, there exist a function $G\in C^\infty_{\mathrm{loc}}(M\setminus \mathcal{S})$ and constants $C_{0},C_{1}>0$ such that
\begin{equation*}
\mathcal{L}_{g} G=0
\end{equation*}
on $M\setminus \mathcal{S}$ and $G\geq C_{0}^{-1}$. Moreover, the function $G(\cdot)$ satisfies
\begin{equation*}
C_{1}^{-1}d_g(x,p)^{-1}\leq G(x)\leq C_{1} d_g(x,p)^{-1}
\end{equation*}
near each $p\in \mathcal{S}$.
\end{lma}

\begin{proof}
This is standard, for example see \cite[Lemma A.1]{ChengLeeTam}.
\end{proof}

Now we are ready to prove Theorem \ref{intro-main}.

\begin{proof}[Proof of Theorem~\ref{intro-main}]

We first observe that under the condition $\mathrm{scal}_g\geq n(n-1)\geq \sigma_{f,g}$ on $M\setminus \mathcal{S}$, if we can show that $\mathrm{scal}_g\equiv\sigma_{f,g}$ on $M\setminus \mathcal{S}$, then $\mu_i\mu_j\equiv 1$ for all $i\neq j$ and hence $\mu_i\equiv 1$ for all $i$, which implies $f^*g_{\mathrm{sph}}=g$ on $M\setminus \mathcal{S}$. 
By \cite[Theorem 2.4]{CecchiniHankeSchick}, 
$f$ is a distance isometry.

In the rest of the argument, thanks to Proposition \ref{prop:dicho}, it suffices to rule out (b) in Proposition \ref{prop:dicho}. In other words, we might assume $g$ in addition satisfy $\mathrm{scal}_g-\sigma_{f,g}\geq \delta_1>0$ on $M\setminus \mathcal{S}$. We now carry out the blow-up argument as in \cite[Proposition 6.2]{LiMantoulidis}. Let $G$ be the Green type function obtained from Lemma~\ref{lma:Green} and denote
\begin{equation*}
\hat g_\e=(1+\e G)^\frac{4}{n-2} g
\end{equation*}
so that $(\hat M,\hat g_\e)=\left(M\setminus\mathcal{S},\hat g_\e \right)$ is a complete non-compact manifold for any $\e>0$ and satisfies
\begin{equation}\label{eqn:scal-low}
\begin{split}
\mathrm{scal}_{\hat g_\e}
= {} & (1+\e G)^{-\frac{n+2}{n-2}}\left( -\frac{4(n-1)}{n-2}\Delta_g +\mathrm{scal}_g\right) (1+\e G) \\[1.5mm]
\geq {} & \e(1+\e G)^{-\frac{n+2}{n-2}}( \mathcal{L}_g+\sigma_{f,g}) G+(1+\e G)^{-\frac{n+2}{n-2}}\mathrm{scal}_g \\[2.5mm]
\geq {} &\frac{\e G}{(1+\e G)^\frac{n+2}{n-2}}\cdot \sigma_{f,g}+\frac{\sigma_{f,g}+\delta_1}{(1+\e G)^\frac{n+2}{n-2}} \\
= {} & \sigma_{f,\hat g_\e}+\frac{\delta_1 }{(1+\e G)^\frac{n+2}{n-2}}\geq 0
\end{split}
\end{equation}
on $\hat M$.

Suppose that $f:(M,g)\to(\mathbb{S}^{n},g_{\mathrm{sph}})$ is $L$-Lipschitz for some constant $L>0$. Then we take $\delta$ to be a positive constant satisfying
$$\delta \leq \min\left\{1,\frac{\delta_1}{6n(n-1)2^\frac{n+2}{n-2}L^2}\right\}.$$
Let $\phi_\delta$ be the map obtained from Lemma \ref{lma:phi-delta} with the chosen  $\delta$,
and define a smooth map $\hat f= \phi_\delta\circ f|_{M\setminus \mathcal{S}}: \hat M^n\to \mathbb{S}^n$.
\bigskip

\noindent
{\it Claim.}
For sufficiently small $\e$,  there exists $\hat f:(\hat M,\hat g_\e)\to (\mathbb{S}^n,g_{\mathrm{sph}})$ which is locally constant at infinity and satisfies
$$\mathrm{scal}_{\hat g_\e}>\sigma_{\hat f,\hat g_\e}$$
on $\mathrm{supp}(\mathrm{d}\hat f)$.

\begin{proof}[Proof of Claim]
If $x\in B_{g}(p_i,\eta)$ for some $p_i\in \mathcal{S}$, then $f(x)\in B_{g_{\mathrm{sph}}}(\hat p_i,L\eta)$. 
Let $V_{\hat p_i,\delta}$ be the open set around each $\hat p_i=f(p_i)\in \hat{\mathcal{S}}=f(\mathcal{S})$ obtained from Lemma~\ref{lma:phi-delta}. Since $\hat{\mathcal{S}}$ is a finite set,  we might choose $\eta(\delta,L)>0$ sufficiently small such that $B_g(p_i,\eta)$ are pairwise disjoint and $B_{g_{\mathrm{sph}}}(\hat p_i,L\eta)\Subset V_{\hat p_i,\delta}$.
Since $\phi_{\delta}$ is constant in $V_{\hat p_i,\delta}$, then $\hat f$ is locally constant near the infinity of $(\hat M,\hat g_\e)$. Clearly, $\hat f$ has the same mapping degree as $f$ and so it is of non-zero degree. It remains to examine the scalar curvature inequality.
We then fix $\e=\e(\delta)$ sufficiently small so that
\begin{equation}\label{eqn:Green-upp}
1+\e G\leq 2
\end{equation}
on $M\setminus \bigcup_{i=1}^N B_g(p_i,\eta)$ by the estimate of $G$ from Lemma~\ref{lma:Green}. 

On $\hat M$, thanks to (i) in Lemma~\ref{lma:phi-delta} and \eqref{eqn:scal-low} we have
\begin{equation*}
\begin{split}
\mathrm{scal}_{\hat g_\e}-\sigma_{\hat f,\hat g_\e}
&\geq  \frac{\delta_1 }{(1+\e G)^\frac{n+2}{n-2}}-\frac{\delta (2+\delta)}{(1+\delta)^2}\sigma_{\hat f,\hat g_\e}.
\end{split}
\end{equation*}
On $M\setminus \bigcup_{i=1}^N B_g(p_i,\eta)$, thanks to \eqref{eqn:Green-upp} we further obtain
\begin{equation*}
\mathrm{scal}_{\hat g_\e}-\sigma_{\hat f,\hat g_\e}
\geq \frac{\delta_1}{2^\frac{n+2}{n-2}}-\frac{\delta (2+\delta)}{(1+\delta)^2}\sigma_{\hat f,\hat g_\e}
\geq \frac{\delta_1}{2^\frac{n+2}{n-2}}-\delta n(n-1) (2+\delta)L^2,\\
\end{equation*}
where we have used the rough estimate from $\|\mathrm{d}\phi_\delta\|_{L^{\infty}}\leq 1+\delta$ and $\|\mathrm{d}f\|_{L^{\infty}}\leq L$. 
From our choice of $\delta$ we have  $\mathrm{scal}_{\hat g_\e}-\sigma_{\hat f,\hat g_\e}>0$ on $M\setminus \bigcup_{i=1}^N B_g(p_i,\eta)$. While for $x_0\in \bigcup_{i=1}^N B_g(p_i,\eta)$, we have $f(x_0)\in B_{g_{\mathrm{sph}}}(\hat p_j,L\eta)$ for some $1\leq j\leq N$ so that $\mathrm{d}\hat f|_{x=x_0}=0$ by our choice of $\eta$, which implies that $\bigcup_{i=1}^N B_g(p_i,\eta)\subset \hat M\setminus \mathrm{supp}(\mathrm{d}\hat f)$. This completes the proof of the claim.
\end{proof}

Since $\hat M=M\setminus\mathcal{S}$ is spin, the claim and Proposition \ref{prop:Zhang-noncpt} implies that $\mathrm{scal}_{\hat g_\e}<0$ somewhere which contradicts with \eqref{eqn:scal-low}. Hence (a) in Proposition~\ref{prop:dicho} must be true. This completes the proof of Theorem \ref{intro-main}.
\end{proof}

\appendix

\section{Variant of Llarull's theorem}

In this section, we recall a generalization of the classical Llarull theorem for complete manifolds, where the scalar curvature lower bound is in term of singular values of a smooth map to the standard sphere. The following is a slight variation of the result of Zhang \cite{Zhang}, see also the work of Listing \cite{Listing}.

\begin{prop}\label{prop:Zhang-noncpt}
Let $(M^n,g)$ be a complete non-compact spin manifold. Suppose that $f:M\to \mathbb{S}^n$ is a smooth map which has non-zero degree and is locally constant near infinity. If the scalar curvature of $g$ satisfies $
\mathrm{scal}_g>\sigma_f$
on $\mathrm{supp}(\mathrm{d}f)$, then $\inf_M\mathrm{scal}_g< 0$.
\end{prop}
\begin{proof}
This is \cite[Theorem 2.1 and 2.2]{Zhang} with slight modification in the proof. Since this is almost identical, we only point out the necessary modifications. We take a compact set $K\Subset M$ so that $f$ is locally constant on $M\setminus K$.
The argument will be divided into the following two cases:

\medskip

{\it Case 1. $n=2m$ is even.} We follow the argument of \cite[Theorem 2.1]{Zhang}. Let $S(TM)=S_+(TM)\oplus S_-(TM)$ be the $\mathbf{Z}_2$-graded Hermitian vector bundle of spinors associated to $(TM,g)$ with the canonical induced Hermitian connection $\nabla^{S(TM)}$, and $S(T\mathbb{S}^n)=S_+(T\mathbb{S}^n)\oplus S_-(T\mathbb{S}^n)$ be the $\mathbf{Z}_2$-graded Hermitian vector bundle of spinors associated to $(T\mathbb{S}^n,g_{\mathrm{sph}})$. Let $E=E_+\oplus E_-$ be the pull-back of $S_+(T\mathbb{S}^n)\oplus S_-(T\mathbb{S}^n)$, $\mathcal{D}^E$ be the canonically defined (twisted by $E$) Dirac operator, and $V$ be the self-adjoint odd endomorphism as defined in \cite[(2.6)]{Zhang}.

Slightly different from the argument in \cite[Theorem 2.1]{Zhang}, we choose $\varphi:M\to [0,1]$ to be a smooth function such that $\varphi=1$ on $(M\setminus \mathrm{supp}(\mathrm{d}f))\cup U_{1/2}$ and $\varphi=0$ on $V_{1/2}$, where $U_{1/2}=\{ x\in \mathrm{supp}(\mathrm{d}f): |\mathrm{d}f(x)|<1/2\}$ and $V_{1/2}=\{ x\in \mathrm{supp}(\mathrm{d}f): \sigma_f>\frac12 n(n-1)\}$ instead. Now we can carry out the same argument by considering the deformed twisted Dirac operator
$$\mathcal{D}^E_\e =\mathcal{D}^E+\e \varphi f^*V.$$
Observe that the choice of $V_{1/2}$ plays no role in relative index Theorem so that \cite[(2.13)]{Zhang} still holds: $\mathrm{ind}(\mathcal{D}^E_{\e,+})\neq 0$. At the same time, we still have \cite[(2.9)]{Zhang}:
\begin{equation*}
(\mathcal{D}^E_\e)^2=(\mathcal{D}^E)^2+ \e c(d\varphi)f^*V +\e \varphi [\mathcal{D}^E,f^*V]+\e^2\varphi^2 f^*(V^2),
\end{equation*}
where $c(\cdot)$ is the Clifford action and $[\cdot,\cdot]$ is the supercommutator. On the other hand, Bochner-Lichnerowicz-Weitzenb\"ock formula infers that
\begin{equation}\label{eqn:Lich}
(\mathcal{D}^E)^2=-\Delta^E +\frac14 \mathrm{scal}_g+\mathcal{R}^E.
\end{equation}
It is well-known that the curvature of the bundle $E$ satisfies
\begin{equation}\label{eqn:R^E}
|\mathcal{R}^Es|\leq \sigma_f |s|
\end{equation}
for all $s\in S_x\otimes E_x$, for example see \cite[Proposition A.1]{BarBrendleHankeWang}. Using \eqref{eqn:R^E} instead of \cite[(2.15)]{Zhang}, we see that for any $x\in V_{1/2}$,
\begin{equation*}
(\mathcal{D}^E_\e)^2+\Delta^E \geq 0,
\end{equation*}
while on $\mathrm{supp}(\mathrm{d}f)\setminus \overline{V_{1/2}}$, \cite[(2.17)]{Zhang} still holds thanks to the choice of $V_{1/2}$, \eqref{eqn:Lich} and \eqref{eqn:R^E}. This will contradict with $\mathrm{ind}(\mathcal{D}^E_{\e,+})\neq 0$ if $\e$ is small enough.

\medskip

{\it Case 2. $n=2m+1$ is odd.} The argument is identical to that of \cite[Theorem 2.4]{Zhang} using the modification made in the even case above.
\end{proof}

\end{document}